\documentclass[11pt,epsf]{article}
\usepackage[dvips]{color}
\usepackage{geometry}
\newcommand{\bed}{\begin{displaymath}}
\newcommand{\eed}{\end{displaymath}}
\newcommand{\bea}{\bed\begin{array}{rl}}
\newcommand{\eea}{\end{array}\eed}
\newcommand{\barray}{\begin{array}{ll}}
\newcommand{\earray}{\end{array}}

\usepackage{color,latexsym,amsfonts,amssymb,amsmath,amsthm}
\usepackage{colordvi,amsfonts,amsmath,epsfig,subfigure,cite}
\usepackage{graphicx}
\usepackage{caption}
\newtheorem{theorem}{Theorem}[section]
\newtheorem{lemma}[theorem]{Lemma}
\newtheorem{remark}{Remark}[section]
\newtheorem{proposition}[theorem]{Proposition}

\newtheorem{definition}{Definition}[section]

\usepackage{fancyhdr}
\begin{document}

\title{Square-mean S-asymptotically $\omega$-periodic solution for a stochastic fractional evolution equation driven by L\'{e}vy noise with piecewise constant argument}
\author{Shufen Zhao$^{a,b}$\thanks{Corresponding author. zsfzx1982@sina.com, 12b312003@hit.edu.cn. This work is supported by the NSF of P.R. China (o.11671113)}
,~~Minghui Song$^{a}$
\\$a$ Department of Mathematics, Harbin Institute of Technology, Harbin 150001, PR China\\
$b$ Department of  Mathematics , Zhaotong University, Zhaotong 657000,PR China}
\maketitle
\begin{abstract}
In this paper, we introduce some concepts of square-mean S-asymptotically $\omega$-periodic stochastic processes.
Using the stochastic analysis method and the Banach contraction mapping principle, we establish the existence and uniqueness results of the mild solution and the square-mean S-asymptotically $\omega$-periodic solution for a semilinear nonautonomous stochastic fractional evolution equation driven by L\'{e}vy noise.


{\bf Mathematics Subject Classfication} 35B15; 34F05; 60H15.

\end{abstract}
\pagestyle{fancy}
\fancyhf{}
\fancyhead[CO]{ Square-mean S-asymptotically $\omega$-periodic solution}
\fancyhead[LE,RO]{\thepage}
\section{Introduction}
\setcounter{equation}{0}
\noindent
In the past decades, there have been many papers dealing with the existence of almost automorphic, asymptotically almost automorphic, almost periodic, asymptotically almost periodic and pseudo almost periodic solutions of various determinate differential systems according to their different applications in different areas ( see e.g. \cite{yoshizawa2012stability,n2013almost,corduneanu1989almost,hino2001almost,henriquez2015almost,diagana2007existence,diagana2009existence}
 and references therein ). In the mean while, the concept of S-asymptotically $\omega$-periodic function is also an interesting topic in mathematical
analysis. The papers \cite{henriquez2008s,pierri2012s,nicola2009note,cuevas2009s,henriquez2016pseudo} are concerned with the existence of the S-asymptotically $\omega$-periodic solutions for the determinate systems in finite dimension and the works \cite{henriquez2008s,cuevas2010existence} are concerned with the S-asymptotically $\omega$-periodic solutions for the determinate systems in infinite dimension.
Especially, Cuevas et al. \cite{cuevas2009s} considered the S-asymptotically $\omega$-periodic solution of the semilinear integro-differential equation of fractional order
\begin{eqnarray*}
\left\{ \begin{array}{ll}
x'(t)=\int_{0}^{t}\frac{(t-s)^{\alpha-2}}{\Gamma(\alpha-1)}A x(s)\mathrm{d}s+f(t,x(t)),\\
x(0)=c_{0}.\end{array} \right.
\end{eqnarray*}
Moreover, in \cite{cuevas2010existence}, Cuevas et al. considered the S-asymptotically $\omega$-periodic solution of the following form,
\begin{eqnarray*}
\left\{ \begin{array}{ll}
x'(t)=\int_{0}^{t}\frac{(t-s)^{\alpha-2}}{\Gamma(\alpha-1)}A x(s)\mathrm{d}s+f(t,x_{t}),\\
x(0)=\psi_{0}\in\mathcal{B},\end{array} \right.
\end{eqnarray*}
 where $\mathcal{B}$ is some abstract phase space,
and in \cite{dimbour2014s}, Dimbour et al. considered the S-asymptotically $\omega$-periodic solutions of the differential equations with piecewise constant argument of the form
\begin{eqnarray*}
\left\{ \begin{array}{ll}
x'(t)=Ax(t)+A_{0}x([t])+g(t,x(t)),\\
x(0)=c_{0}.\end{array} \right.
\end{eqnarray*}
 Note that the effect of the noises is unavoidable in the study of some natural sciences as well as man-made phenomena such as ecology, biology, finance markets, and engineering and other fields. Since the properties of the L\'{e}vy processes have useful information in explaining some drastic changes in nature and various scientific fields \cite{applebaum2009levy,rong2006theory,fu2012almost}, for the potential applications in theory and applications, stochastic fractional evolution equations driven by L\'{e}vy noise have been attracting researchers' increasing interest \cite{peszat2007stochastic,liu2014almost,li2015weighted}. It is natural to extend the known results on determinate systems to stochastic systems. So in this paper, we investigate some properties of a semilinear nonautonomous stochastic fractional evolution equation driven by L\'{e}vy noise with piecewise constant argument of the following from
\begin{eqnarray}\label{1}
 \left\{ \begin{array}{ll}
\mathrm{d}x(t)=\int_{0}^{t}\frac{(t-s)^{\alpha-2}}{\Gamma(\alpha-1)}A x(s)\mathrm{d}s\mathrm{d}t+f(t,x([t]),x(t))\mathrm{d}t+g(t,x([t]),x(t))\mathrm{d}w(t)\\
\qquad~~+\int_{|u|_{U}<1}F(t,x(t^{-}),u)\tilde{N}(\mathrm{d}t,\mathrm{d}u)+\int_{|u|_{U}\geq1}G(t,x(t^{-}),u)N(\mathrm{d}t,\mathrm{d}u),\\
x(0)=c_{0},
\end{array} \right.
\end{eqnarray}
where $x(\cdot)$ takes value in a real separable Hilbert space $H,$ $1<\alpha<2,$ $A$ is a linear densely defined operator of sectorial type on $H, $ with domain $D(A).$
The convolution integral in (\ref{1}) is the Riemann-Liouville fractional integral.
To the best of our knowledge, the existence and uniqueness results for the Cauchy problem (\ref{1}) have not been investigated and that is the main motivation of this paper. This paper is devoted to the existence of the mild solution and the square-mean S-asymptotically $\omega$-periodic solution of (\ref{1}). In order to correspond to the effect of L\'{e}vy noise, we introduce some new concepts of Poisson square-mean S-asymptotically $\omega$-periodic stochastic processes.

This article is composed of four sections. In Section 2, we introduce notations, definitions, and preliminary facts which are useful throughout the article. In Section 3, we prove the existence of the mild solution for problem (\ref{1}) by using the successive approximation. Finally, in Section 4, we show the existence and uniqueness of the $S$-asymptotically $\omega$-periodic solution for (\ref{1}) by using the Banach contraction mapping principle.
\section{Preliminaries}
Let $(\Omega,\mathcal{F},P)$ be a complete probability space equipped with some filtration $\{\mathcal{F}_{t}\}_{t\geq0}$ satisfying the usual conditions, $(H,|\cdot|)$ and $(U,|\cdot|)$ are real separable Hilbert spaces. $\mathcal{L}(U,H)$ denote the space of all bounded linear operators from $U$ to $H$ which with the usual operator norm $\|\cdot\|_{\mathcal{L}(U,H)}$ is a Banach space. $L^{2}(P,H)$ is the space of all $H$-valued random variables $X$ such that $E|X|^{2}=\int_{\Omega}|X|^{2}\mathrm{d}P<\infty.$ For $X\in L^{2}(P,H),$ let $\|X\|:=(\int_{\Omega}\|X\|^{2}\mathrm{d}P)^{1/2},$ it is well known that $(L^{2}(P,H),\|\cdot\|)$ is a Hilbert space. In the following discussion, we always consider the L\'{e}vy processes that are $U$-valued.
\subsection{L\'{e}vy process}
Let $L$ is a L\'{e}vy process on $U,$ we write $\Delta L(t)=L(t)-L(t^{-})$ for all $t\geq 0.$  We define a counting Poisson random measure $N$ on $(U-\{0\})$ through $$N(t,O)=\sharp\{0\leq s\leq t:~\Delta L(s)(\omega')\in O\}=\Sigma_{0\leq s\leq t}\chi_{O}(\Delta L(s)(\omega'))$$ for any Borel set $O$ in $(U-\{0\}),$ $\chi_{O}$ is the indicator function. We write $\nu(\cdot)=E(N(1,\cdot))$ and call it the intensity measure associated with $L$. We say that a Borel set $O$ in $(U-\{0\}),$ is bounded below if $0\in \bar{O}$ where $\bar{O}$ is the closure of $O.$ If $O$ is bounded below, then $N(t,O)<\infty$ almost surely for all $t\geq 0$ and $(N(t,O),~t\geq0)$ is a Poisson process with intensity $\nu(O).$ So $N$ is called Poisson random measure. For each $t\geq 0$ and $O$ bounded below, the associated compensated Poisson random measure $\tilde{N}$ is defined by $\tilde{N}(t,O)=N(t,O)-t\nu(O)$ (see \cite{applebaum2009levy,holden2009stochastic}).
\begin{proposition}(see \cite{applebaum2009levy})
(L\'{e}vy-It\^{o} decomposition).
If $L$ is a $U$-valued L\'{e}vy process, then there exist $a\in U,$ a $U$-valued Wiener process $w$ with covariance operator $Q,$ the so-called $Q$-wiener process, and an independent Poisson random measure $N$ on $\mathbf{R}^{+}\times (U-\{0\})$ such that, for each $t\geq0,$
\begin{equation}\label{ito}
  L(t)=at+w(t)+\int_{|u|_{U}<1}u\tilde{N}(t,\mathrm{d}u)+\int_{|u|_{U}\geq1}uN(t,\mathrm{d}u),
\end{equation}
where the Poisson random measure $N$ has the intensity measure $\nu$ which satisfies
$\int_{U}(|y|_{U}^{2}\wedge 1)\nu(\mathrm{d}y)<\infty$ and $\tilde{N}$ is the compensated Poisson random measure of $N.$
\end{proposition}
The detail properties of L\'{e}vy process and $Q$-Wiener processes, we refer the readers to \cite{albeverio2002stochastic} and \cite{da2014stochastic}. Throughout the paper, we assume that the covariance operator $Q$ of $w$ is of trace class, i.e. $TrQ<\infty$ and the L\'{e}vy process $L$ is defined on the filtered probability space $(\Omega,\mathcal{F},P,(\mathcal{F}_{t})_{t\in\mathbf{R}^{+}}).$ Throughout the paper we also denote by $b:=\int_{|x|_{U}\geq 1}\nu(\mathrm{d}x).$
\subsection{The definition of square-mean S-asymptotically $\omega$-periodic stochastic process }
\begin{definition}(see \cite{liu2014almost})
A stochastic process $x:\mathbf{R}\rightarrow L^{2}(P,H)$ is said to be $L^{2}$-continuous if for any $s\in\mathbf{R},$ $\lim_{t\rightarrow s}\|x(t)-x(s)\|^{2}=0.$ It is $L^{2}$-bounded if $\sup_{t\in\mathbf{R}}\|x(t)\|<\infty.$
\end{definition}
\begin{definition}
\begin{enumerate}
  \item [(1)] An $L^{2}$-continuous stochastic process $x:\mathbf{R}^{+} \rightarrow L^{2}(P,H)$ is said to be square-mean S-asymptotically $\omega$-periodic if there exists $\omega>0$ such that $\lim_{t\rightarrow\infty}\|x(t+\omega)-x(t)\|=0.$ The collection of all S-asymptotically $\omega$-periodic stochastic processes $x:\mathbf{R}^{+}\rightarrow L^{2}(P,H)$ is denoted by $SAP_{\omega}(L^{2}(P,H)).$
      \item[(2)] A function $f:\mathbf{R}^{+} \times L^{2}(P,H)\rightarrow L^{2}(P,H),$ $(t,X)\mapsto f(t,X)$ is said to be square-mean S-asymptotically $\omega$-periodic in $t$ for each $X\in L^{2}(P,H)$ if $f$ is continuous in the following sense
      $$\|f(t,X)-f(t',X')\|\rightarrow0,~(t',X')\rightarrow(t,X)$$ and
      $$\lim_{t\rightarrow\infty}\|f(t+\omega,X)-f(t,X)\|\rightarrow0$$ for each $X\in L^{2}(P,H).$
  \item [(3)] A function $g:\mathbf{R}^{+} \times L^{2}(P,H)\rightarrow \mathcal{L}(U,L^{2}(P,H)),$ $(t,X)\mapsto g(t,X)$ is said to be square-mean S-asymptotically $\omega$-periodic in $t$ for each $X\in L^{2}(P,H)$ if $g$ is continuous in the following sense $$E\|(g(t,X)-g(t',X'))Q^{1/2}\|^{2}_{\mathcal{L}(U,L^{2}(P,H))}\rightarrow0,~(t',X')\rightarrow(t,X)$$
      and $$\lim_{t\rightarrow\infty}E\|(g(t+\omega,X)-g(t,X))Q^{1/2}\|^{2}_{\mathcal{L}(U,L^{2}(P,H))}=0$$ for each $X\in L^{2}(P,H).$
  \item [(4)] A function $F:\mathbf{R}^{+} \times L^{2}(P,H)\times U\rightarrow L^{2}(P,H),$ $(t,X,u)\mapsto F(t,X,u)$ with $ \int_{U}\|F(t,\phi,u)\|^{2}\nu(\mathrm{d}u)<\infty$ is said to be Poisson square-mean S-asymptotically $\omega$-periodic in $t$ for each $X\in L^{2}(P,H)$ if $F$ is continuous in the following sense $$\int_{U}\|F(t,X,u)-F(t',X',u)\|^{2}\nu(\mathrm{d}u)\rightarrow0,~(t',X')\rightarrow(t,X)$$ and that $$lim_{t\rightarrow\infty}\int_{U}\|F(t+\omega,X,u)-F(t,X,u)\|^{2}\nu(\mathrm{d}u)=0$$ for each $X\in L^{2}(P,H).$
\end{enumerate}
\end{definition}
\begin{remark}
Any square-mean S-asymptotically $\omega$-periodic process $x(t)$ is $L^{2}$-bounded and, by \cite{henriquez2008s}, $SAP_{\omega}(L^{2}(P,H))$ is a Banach space when it is equipped with the norm $$\|x\|_{\infty}:=\sup_{t\in\mathbf{R}^{+}}\|x(t)\|=\sup_{t\in\mathbf{R}^{+}}(E|x(t)|^{2})^{\frac{1}{2}}.$$
\end{remark}
For the sequel, we introduce some definitions about square-mean S-asymptotically $\omega$-periodic stochastic processes with parameters.
\begin{definition}
\begin{enumerate}
  \item [(1)]A function $f:\mathbf{R}^{+}\times L^{2}(P,H)\rightarrow L^{2}(P,H)$ is said to be uniformly square-mean S-asymptotically $\omega$-periodic in $t$ on bounded sets if for every bounded set $K$ of $L^{2}(P,H)$, we have $\lim_{t\rightarrow\infty}\|f(t+\omega,X)-f(t,X)\|=0$ uniformly on $x\in K.$
  \item [(2)] A function $g:\mathbf{R}^{+} \times L^{2}(P,H)\rightarrow \mathcal{L}(U,L^{2}(P,H)),$ is said to be uniformly square-mean S-asymptotically $\omega$-periodic on bounded sets if for every bounded set $K$ of $L^{2}(P,H)$, we have
    $$
        \lim_{t\rightarrow\infty}E\|(g(t+\omega,X)-g(t,X))Q^{1/2}\|^{2}_{\mathcal{L}(U,L^{2}(P,H))}=0
 $$
    uniformly on $X\in K.$
  \item [(3)] A function $F:\mathbf{R}^{+} \times L ^{2}(P,H)\times U\rightarrow L^{2}(P,H)$ with $ \int_{U}\|F(t,\phi,u)\|^{2}\nu(\mathrm{d}u)<\infty,$ is said to be uniformly Poisson square-mean S-asymptotically $\omega$-periodic in $t$ on bounded sets if
     $$
        \lim_{t\rightarrow\infty}\int_{U}\|F(t+\omega,X,u)-F(t,X,u)\|^{2}\nu(\mathrm{d}u)=0
      $$
     uniformly on $X\in K.$
\end{enumerate}
\end{definition}
\begin{lemma}\label{l1}
Let $f:\mathbf{R}\times L^{2}(P,H)\rightarrow L^{2}(P,H),~(t,X)\rightarrow f(t,X)$ be uniformly  square-mean S-asymptotically $\omega$-periodic in $t$ on bounded sets of $ L^{2}(P,H),$ and assume that $f$ satisfies the Lipschitz condition in the sense $\|f(t,Y)-f(t,Z)\|^{2}\leq L\|Y-Z\|^{2}$ for all $Y,Z\in L^{2}(P,H)$ and $t\in \mathbf{R}, $ where $L$ is independent of $t.$ Then for any square-mean S-asymptotically $\omega$-periodic process $Y:\mathbf{R}\rightarrow L^{2}(P,H),$ the stochastic process $F:\mathbf{R}\rightarrow L^{2}(P,H)$ given by $F(t):=f(t,Y(t))$ is square-mean S-asymptotically $\omega$-periodic.
\end{lemma}
\begin{proof}
Since $Y(t)\in SAP_{\omega}( L^{2}(P,H))$, the range of $Y(t)$ is a bounded set in $ L^{2}(P,H).$
Then $\lim_{t\rightarrow\infty}\|F(t+\omega,Y(t+\omega))-F(t,Y(t+\omega))\|=0.$
For any $\epsilon>0,$ $\exists ~T(\epsilon),$ such that $\|F(t+\omega,Y(t+\omega))-F(t,Y(t+\omega))\|<\frac{\epsilon}{2}$ and $\|Y(t+\omega)-Y(t)\|<\frac{\epsilon}{2L}.$
We get
\begin{eqnarray*}
  &&\|F(t+\omega)-F(t)\|\\&&\leq\|F(t+\omega,Y(t+\omega))-F(t,Y(t))\|\\
  &&=\|F(t+\omega,Y(t+\omega))-F(t,Y(t+\omega))\|+\|F(t,Y(t+\omega))-F(t,Y(t))\|\\
  &&\leq\epsilon,
\end{eqnarray*}
which completes the proof.
\end{proof}
\begin{lemma}\label{l2}
Let $F:\mathbf{R}^{+} \times L^{2}(P,H)\times U\rightarrow  L^{2}(P,H)$ be uniformly Poisson square-mean S-asymptotically $\omega$-periodic in $t$ on bounded sets of $L^{2}(P,H),$ and $F$ satisfies Lipschitz condition in the sense $$\int_{U}\|F(t,Y,u)-F(t,Z,u)\|^{2}\nu(\mathrm{d}u)\leq L\|Y-Z\|^{2}$$ for all $Y,Z\in L^{2}(P,H)$ and $t\in \mathbf{R}, $ where $L$ is independent of $t.$ Then for any square-mean S-asymptotically $\omega$-periodic process $Y(t):\mathbf{R}\rightarrow L^{2}(P,H),$ the stochastic process $\tilde{F}:\mathbf{R}\times U\rightarrow   L^{2}(P,H)$ given by $\tilde{F}(t,u):=F(t,Y(t),u)$ satisfies
 $$
  \lim_{t\rightarrow\infty }\int_{U}\|\tilde{F}(t+\omega,u)-\tilde{F}(t,u)\|^{2}\nu(\mathrm{d}u)=0.
$$
\end{lemma}
\begin{proof}
Since $F$ is uniformly Poisson square-mean S-asymptotically $\omega$-periodic in $t$ on bounded sets of $L^{2}(P,H),$ and
$Y\in SAP_{\omega}(L^{2}(P,H),$ the range $\mathcal{R}(Y)$ of $Y(t)$ is a bounded set in $L^{2}(P,H),$
then $$\lim_{t\rightarrow\infty}\int_{U}\|F(t+\omega,Y,u)-F(t,Y,u)\|^{2}\nu(\mathrm{d}u)=0$$ uniformly for $Y\in \mathcal{R}(Y).$ For any $\epsilon>0,$ we can find $T(\epsilon)>0$ such that when $t\geq T(\epsilon),$ we have $\int_{U}\|F(t+\omega,\bar{Y},u)-F(t,\bar{Y},u)\|^{2}\nu(\mathrm{d}u)<\epsilon/4,~\forall~\bar{Y}\in \mathcal{R}(Y)$ and $\|Y(t+\omega)-Y(t)\|^{2}<\frac{\epsilon}{2L}.$
 Note that
\begin{eqnarray*}
&&\tilde{F}(t+\omega,u)-\tilde{F}(t,u)\\
&&=F(t+\omega,Y(t+\omega),u)-F(t+\omega,Y(t),u)+F(t+\omega,Y(t),u)-F(t,Y(t),u),
\end{eqnarray*}
so for the above $\epsilon,$ when $t\geq T(\epsilon),$ we have
\begin{eqnarray*}
&&\int_{U}\|\tilde{F}(t+\omega,u)-\tilde{F}(t,u)\|^{2}\nu(\mathrm{d}u)\\
&&\leq 2\int_{U}\|F(t+\omega,Y(t+\omega),u)-F(t,Y(t+\omega),u)\|^{2}\nu(\mathrm{d}u)\\
&&~+2\int_{U}\|F(t,Y(t+\omega),u)-F(t,Y(t),u)\|^{2}\nu(\mathrm{d}u)\\
&&\leq \epsilon/2+2L\|Y(t+\omega)-Y(t)\|^{2}\\
&&\leq \epsilon.
\end{eqnarray*}
We get the desired result.
\end{proof}
\subsection{Sectorial operators}
We recall some definitions about sectorial operators which have been studied well in the past decades, for details, see \cite{haase2006functional,lunardi2012analytic}.
\begin{definition}
Let $\mathbf{X}$ be an Banach space, $A:D(A)\subseteq \mathbf{X}\rightarrow \mathbf{X}$ is a close linear operator. $A$ is said to be a sectorial operator of type $\mu$ and angle $\theta$ if there exist $0<\theta<\pi/2,~M>0$ and $\mu\in\mathbf{R}$ such that the resolvent $\rho(A)$ of $A$ exists outside the sector $\mu+S_{\theta}=\{\mu+\lambda:\lambda\in\mathbf{C} ,|arg(-\lambda)|<\theta\}$ and $\|(\lambda-A)^{-1}\|\leq \frac{M}{|\lambda-\mu|}$ when $\lambda$ does not belong to $\mu+S_{\theta}.$
\end{definition}
\begin{definition}(see \cite{dos2010asymptotically})
Let $A$ be a closed and linear operator with domain $D(A)$ defined on a Banach space $\mathbf{X}.$ We call $A$ the generator of a solution operator if there exist $\mu\in \mathbf{R}$ and a strongly continuous function $S_{\alpha}:~\mathbf{R}^{+}\rightarrow\mathcal{L}(\mathbf{X},\mathbf{X})$ such that $\{\lambda^{\alpha}:~Re(\lambda)>\mu\}\subset\rho(A)$ and $\lambda^{\alpha-1}(\lambda^{\alpha}-A)^{-1}x=\int_{0}^{\infty}e^{-\lambda t}S_{\alpha}(t)\mathrm{d}t,$ $Re(\lambda)>\mu,$ $x\in \mathbf{X}.$ In this case, $S_{\alpha}(\cdot)$ is called the solution operator generated by $A.$
\end{definition}
If $A$ is sectorial of type $\mu$ with $1<\theta<\pi(1-\frac{\alpha}{2}),$ then $A$ is the generator of a solution operator given by $S_{\alpha}(t)=\frac{1}{2\pi i}\int_{\gamma }e^{\lambda t}\lambda ^{\alpha-1}(\lambda^{\alpha}-A)^{-1}\mathrm{d}\lambda,$ where $\gamma$ is a suitable path lying outside the sector $\mu+S_{\theta}$ \cite{bajlekova2001fractional}. Cuesta \cite{cuesta2007asymptotic} showed that if $A$ is a sectorial operator of type $\mu<0,$ for some $M>0$ and $0<\theta<\pi(1-\frac{\pi}{2}),$ there is $C>0$ such that
\begin{equation}\label{sem}
  \|S_{\alpha}(t)\|\leq\frac{CM}{1+|\mu|t^{\alpha}},~t\geq 0.
\end{equation}
\begin{definition}
A stochastic process $\{x(t),~t\in [0,T]\},$ $0\leq T<\infty$ is said to be a mild solution to (\ref{1}) if
\begin{enumerate}
  \item [(i)] $x(t)$ is $\mathcal{F}_{t}$-adapted and has C\`{a}dl\`{a}g paths on $t\geq0$ almost surely,
  \item  [(ii)] $x(t)$ satisfies the following stochastic integral
equation
\begin{eqnarray*}\label{sol}
x(t)
&=&S_{\alpha}(t)c_{0}+\int_{0}^{t}S_{\alpha}(t-s)f(s,x([s]),x(s))\mathrm{d}s\\&&~+\int_{0}^{t}S_{\alpha}(t-s)g(s,x([s]),x(s))\mathrm{d}w(s)
\\&&~+\int_{0}^{t}\int_{|u|<1}S_{\alpha}(t-s)F(s,x(s^{-}),u)\tilde{N}(\mathrm{d}s,\mathrm{d}u)\\&&~+\int_{0}^{t}\int_{|u|\geq1}S_{\alpha}(t-s)G(s,x(s^{-}),u)N(\mathrm{d}s,\mathrm{d}u).
\end{eqnarray*}
\end{enumerate}

\end{definition}
\section{The existence of the mild solution for (\ref{1})}
In order to establish our main result, we impose the following conditions.
\begin{enumerate}
\item [(H1)] $A$ is a sectorial operator of type $\mu<0$ and angle $\theta$ with $0\leq \theta\leq \pi(1-\alpha/2).$
  \item [(H2)] $f:\mathbf{R}^{+}\times  L^{2}(P,H)\times L^{2}(P,H)\rightarrow L^{2}(P,H),$ $g:~\mathbf{R}^{+}\times L^{2}(P,H)\times L^{2}(P,H)\rightarrow \mathcal{L}(U,L^{2}(P,H))$ are jointly measurable.~ $F:\mathbf{R}^{+}\times L^{2}(P,H)\times U\rightarrow L^{2}(P,H),$ $G:\mathbf{R}^{+}\times L^{2}(P,H)\times U\rightarrow L^{2}(P,H)$ are jointly measurable and $$\int_{|u|_{U}<1}\|F(t,x,u)\|\nu(\mathrm{d}u)<\infty,~\int_{|u|_{U}<1}\|G(t,x,u)\|\nu(\mathrm{d}u)<\infty.$$ For all $t\in \mathbf{R}^{+},$
      \begin{eqnarray*}
    \|f(t,x,y)-f(t,x_{1},y_{1})\|^{2}\leq L(\|x-x_{1}\|^{2}+\|y-y_{1}\|^{2}),\\
   E\|(g(t,x,y)-g(t,x_{1},y_{1}))Q^{1/2}\|^{2}_{\mathcal{L}(U,L^{2}(P,H))}\leq L(\|x-x_{1}\|^{2}+\|y-y_{1}\|^{2}),\\
    \int_{|u|_{U}<1}\|F(t,x,u)-F(t,z,u)\|^{2}\nu(\mathrm{d}u)\leq L\|x-z\|^{2},\\
      \int_{|u|_{U}\geq1}\|G(t,y,u)-G(t,z,u)\|^{2}\nu(\mathrm{d}u)\leq L\|y-z\|^{2},
      \end{eqnarray*} for some constant $L>0$ independent of $t.$
   \end{enumerate}
\begin{theorem}
If (H1)-(H2) hold, then the Cauchy problem (\ref{1}) has a unique mild solution.
\end{theorem}
\begin{proof}Set $x_{0}\equiv S_{\alpha}(t)c_{0},$ and for $n=1,2,...,$ $\forall~ T\in (0,\infty),$ we define the Picard iterations as follows:
\begin{eqnarray}\label{p1}
 \nonumber x_{n}(t)&=&S_{\alpha}(t)c_{0}+\int_{0}^{t}S_{\alpha}(t-s)f(s,x_{n-1}([s]),x_{n-1}(s))\mathrm{d}s
 \nonumber\\&&+\int_{0}^{t}S_{\alpha}(t-s)g(s,x_{n-1}([s]),x_{n-1}(s))\mathrm{d}w(s)
  \nonumber\\&&+\int_{0}^{t}\int_{|u|<1}S_{\alpha}(t-s)F(s,x_{n}(s^{-}),u)\tilde{N}(\mathrm{d}s,\mathrm{d}u)
 \nonumber\\&&+\int_{0}^{t}\int_{|u|\geq1}S_{\alpha}(t-s)G(s,x_{n}(s^{-}),u)N(\mathrm{d}s,\mathrm{d}u)
\end{eqnarray}
for $t\in[0,T].$ Obviously, $x_{0}(\cdot)\in L^{2}(P,H),$
\begin{eqnarray*}
&&\|x_{n}(t)\|^{2}\\
&&=\|S_{\alpha}(t)c_{0}+\int_{0}^{t}S_{\alpha}(t-s)f(s,x_{n-1}([s]),x_{n-1}(s))\mathrm{d}s\\
&&~+\int_{0}^{t}S_{\alpha}(t-s)g(s,x_{n-1}([s]),x_{n-1}(s))\mathrm{d}w(s)\\&&~+\int_{0}^{t}\int_{|u|<1}S_{\alpha}(t-s)F(s,x_{n-1}(s^{-}),u)\tilde{N}(\mathrm{d}s,\mathrm{d}u)
 \nonumber\\&&~+\int_{0}^{t}\int_{|u|\geq1}S_{\alpha}(t-s)G(s,x_{n-1}(s^{-}),u)N(\mathrm{d}s,\mathrm{d}u)\|^{2}\\
&&\leq5(CM)^{2}\|c_{0}\|^{2}+5\|\int_{0}^{t}S_{\alpha}(t-s)f(s,x_{n-1}([s]),x_{n-1}(s))\mathrm{d}s\|^{2}\\
&&~+5\int_{0}^{t}\|S_{\alpha}(t-s)\|^{2}E\|g(s,x_{n-1}([s]),x_{n-1}(s))Q^{1/2}\|_{\mathcal{L}(U,L^{2}(P,H))}^{2}\mathrm{d}s\\&&~+5\|\int_{0}^{t}\int_{|u|_{U}<1}S_{\alpha}(t-s)F(s,x_{n-1}(s^{-}),u)\tilde{N}(\mathrm{d}s,\mathrm{d}u)\|^{2}\\
&&~+5\|\int_{0}^{t}\int_{|u|\geq1}S_{\alpha}(t-s)G(s,x_{n-1}(s^{-}),u)N(\mathrm{d}s,\mathrm{d}u)\\&&~+\int_{0}^{t}\int_{|u|\geq1}S_{\alpha}(t-s)G(s,x_{n-1}(s^{-}),u)\nu(\mathrm{d}u)\mathrm{d}s\|^{2}
\\
&&\leq5(CM)^2\|c_{0}\|^{2}+5(CM)^{2}\int_{0}^{t}\frac{1}{1+|\mu|(t-s)^{\alpha}}\mathrm{d}s\\
&&~\times\int_{0}^{t}\frac{1}{1+|\mu|(t-s)^{\alpha}}L(\|x_{n-1}([s])\|^{2}+\|x_{n-1}(s)\|^{2})\mathrm{d}s\\
&&~+5(CM)^{2}\int_{0}^{t}\frac{1}{1+|\mu|^{2}(t-s)^{2\alpha}}L(\|x_{n-1}([s])\|^{2}+\|x_{n-1}(s)\|^{2})\mathrm{d}s\\
&&~+5(CM)^{2}L\int_{0}^{t}\|x_{n-1}(s^{-})\|^{2}\mathrm{d}s+10(CM)^{2}L\int_{0}^{t}\|x_{n-1}(s)\|^{2}\mathrm{d}s\\
&&~+10(CM)^{2}\frac{|\mu|^{-1/\alpha}\pi}{\alpha\sin(\pi/\alpha)}bL\int_{0}^{t}\|x_{n-1}(s^{-})\|^{2}\mathrm{d}s
\end{eqnarray*}
\begin{eqnarray*}
&&\leq5(CM)^2\|c_{0}\|^{2}+5(CM)^{2}L[\frac{|\mu|^{-1/\alpha}\pi}{\alpha\sin(\pi/\alpha)}+1]\times\int_{0}^{t}(\|x_{n-1}([s])\|^{2}+\|x_{n-1}(s)\|^{2})\mathrm{d}s\\
&&~+5(CM)^{2}L\int_{0}^{t}\|x_{n-1}(s^{-})\|^{2}\mathrm{d}s+10(CM)^{2}L\int_{0}^{t}\|x_{n-1}(s)\|^{2}\mathrm{d}s\\
&&~+10(CM)^{2}\frac{|\mu|^{-1/\alpha}\pi}{\alpha\sin(\pi/\alpha)}bL\int_{0}^{t}\|x_{n-1}(s^{-})\|^{2}\mathrm{d}s.
\end{eqnarray*}
So we get
 \begin{eqnarray*}
&&\sup_{0\leq s\leq t} \|x_{n}(s)\|^{2}\\&&\leq5(CM)^{2}\|c_{0}\|^{2}\\&&~~+5(CM)^{2}L(2\frac{|\mu|^{-1/\alpha}\pi}{\alpha\sin(\pi/\alpha)}+5+2b\frac{|\mu|^{-1/\alpha}\pi}{\alpha\sin(\pi/\alpha)})\int_{0}^{t}\sup_{0\leq\theta\leq s}\|x_{n-1}(\theta)\|^{2}\mathrm{d}s.
  \end{eqnarray*}
Then for any arbitrary positive integer $\tilde{k},$ we have
\begin{eqnarray*}
 &&\max_{1\leq n\leq \tilde{k}} \sup_{0\leq s\leq t} \|x_{n}(s)\|^{2}\\
 &&\leq 3(CM)^{2}\|c_{0}\|^{2}+5(CM)^{2}L(2\frac{|\mu|^{-1/\alpha}\pi}{\alpha\sin(\pi/\alpha)}+5+2b\frac{|\mu|^{-1/\alpha}\pi}{\alpha\sin(\pi/\alpha)})\\
 &&\quad\times\int_{0}^{t}\max_{1\leq n\leq \tilde{k}}\sup_{0\leq\theta\leq s}\|x_{n-1}(\theta)\|^{2}\mathrm{d}s.
  \end{eqnarray*}
  If we let $c_{1}=5(CM)^{2}\|c_{0}\|^{2},~c_{2}=5(CM)^{2}L(2\frac{|\mu|^{-1/\alpha}\pi}{\alpha\sin(\pi/\alpha)}+5+2b\frac{|\mu|^{-1/\alpha}\pi}{\alpha\sin(\pi/\alpha)}),$ then by the Gronwall inequality, we get
  $\max_{1\leq n\leq \tilde{k}} \sup_{0\leq s\leq t} \|x_{n}(s)\|^{2}\leq c_{1}e^{c_{2}t}.$
  Due to the arbitrary of $\tilde{k},$ we have
  \begin{equation}\label{squa}
    \sup_{0\leq s\leq t} \|x_{n}(s)\|^{2}\leq c_{1}e^{c_{2}T}.
  \end{equation}
Note that
  \begin{eqnarray}\label{fir1}
    &&\|x_{1}(t)-x_{0}(t)\|^{2}\\
   \nonumber &&\leq 4(CM)^{2}\|c_{0}\|^{2}L(\frac{|\mu|^{-1/\alpha }\pi}{\alpha\sin(\pi/\alpha)})^{2}+16(CM)^{2}\|c_{0}\|^{2}L(\frac{|\mu|^{-2/\alpha}\pi}{2\alpha\sin(\pi/2\alpha)})\\
     \nonumber &&\quad+8(CM)^{2}\|c_{0}\|^{2}L(\frac{|\mu|^{-1/\alpha }\pi}{\alpha\sin(\pi/\alpha)})^{2}b
    =\tilde{C},
   \end{eqnarray}
   we claim that for $n\geq0,$\begin{eqnarray}\label{indu}
                                   \|x_{n+1}(t)-x_{n}(t)\|^{2}\leq \frac{\tilde{C}(\tilde{M}t)^{n}}{n!},~0\leq t\leq T,
                                  \end{eqnarray}
  where $\tilde{M}=4(CM)^{2}L[\frac{2|\mu|^{-1/\alpha}\pi}{\alpha\sin(\pi/\alpha)}+5+2b\frac{|\mu|^{-1/\alpha }\pi}{\alpha\sin(\pi/\alpha)}],$ we will show this claim by induction. In view of (\ref{fir1}), we see that (\ref{indu}) holds when $n=0.$ By assuming that (\ref{indu}) holds for some $n\geq 0,$ we shall show that (\ref{indu}) still holds for $n+1.$ Note that
 \begin{eqnarray*}
 && \|x_{n+2}(t)-x_{n+1}(t)\|^{2}\\
  &&\leq 4\|\int_{0}^{t}S_{\alpha}(t-s)[f(s,x_{n+1}([s]),x_{n+1}(s))-f(s,x_{n}([s]),x_{n}(s))]\mathrm{d}s\|^{2}\\
  &&~+4\|\int_{0}^{t}S_{\alpha}(t-s)[g(s,x_{n+1}([s]),x_{n+1}(s))-g(s,x_{n}([s]),x_{n}(s))]\mathrm{d}w(s)\|^{2}\\
   &&~+4\|\int_{0}^{t}\int_{|u|<1}S_{\alpha}[F(s,x_{n+1}(s^{-}),u)-F(s,x_{n}(s^{-}),u)]\tilde{N}(\mathrm{d}s,\mathrm{d}u)\|^{2}\\
   &&~+4\|\int_{0}^{t}\int_{|u|\geq1}S_{\alpha}[G(s,x_{n+1}(s^{-}),u)-G(s,x_{n}(s^{-}),u)]N(\mathrm{d}s,\mathrm{d}u)\|^{2}\\
  &&\leq4(CM)^{2}L(\frac{|\mu|^{-1/\alpha}\pi}{\alpha\sin(\pi/\alpha)}+1)\\&&~~~\times\int_{0}^{t}[\|x_{n+1}([s])-x_{n}([s])\|^{2}+\|x_{n+1}(s)-x_{n}(s)\|^{2}]\mathrm{d}s\\
  &&~+4(CM)^{2}L\int_{0}^{t}\frac{1}{1+|\mu|^{2}(t-s)^{2\alpha}}\|x_{n+1}(s^{-})-x_{n}(s^{-})\|^{2}\mathrm{d}s\\
  &&~+8\|\int_{0}^{t}\int_{|u|\geq1}S_{\alpha}(t-s)[G(s,x_{n+1}(s^{-}),u)-G(s,x_{n}(s^{-}),u)]\tilde{N}(\mathrm{d}s,\mathrm{d}u)\|^{2}\\
  &&~+8\|\int_{0}^{t}\int_{|u|\geq1}S_{\alpha}(t-s)[G(s,x_{n+1}(s^{-}),u)-G(s,x_{n}(s^{-}),u)]\nu(\mathrm{d}u)\mathrm{d}s\|^{2}\\
  &&\leq4(CM)^{2}L(\frac{|\mu|^{-1/\alpha}\pi}{\alpha\sin(\pi/\alpha)}+1)[\int_{0}^{t}\frac{\tilde{C}(\tilde{M}[s])^{n}}{n!}\mathrm{d}s+\int_{0}^{t}\frac{\tilde{C}(\tilde{M}s)^{n}}{n!}\mathrm{d}s]\\
  &&~+12(CM)^{2}L\int_{0}^{t}\frac{1}{1+|\mu|^{2}(t-s)^{2\alpha}}\frac{\tilde{C}(\tilde{M}s)^{n}}{n!}\mathrm{d}s\\
  &&~+8(CM)^{2}b\int_{0}^{t}\frac{1}{1+|\mu|(t-s)^{\alpha}}\mathrm{d}s\int_{0}^{t}\|x_{n+1}(s^{-})-x_{n}(s^{-})\|^{2}\mathrm{d}s\end{eqnarray*}\begin{eqnarray*}
  &&\leq8(CM)^{2}L(\frac{|\mu|^{-1/\alpha}\pi}{\alpha\sin(\pi/\alpha)}+1)\int_{0}^{t}\frac{\tilde{C}(\tilde{M}s)^{n}}{n!}\mathrm{d}s\\
  &&~+12(CM)^{2}L\int_{0}^{t}\frac{\tilde{C}(\tilde{M}s)^{n}}{n!}\mathrm{d}s\\
  &&~+8(CM)^{2}bL\frac{|\mu|^{-1/\alpha }\pi}{\alpha\sin(\pi/\alpha)}\int_{0}^{t}\frac{\tilde{C}(\tilde{M}s)^{n}}{n!}\mathrm{d}s\\
  &&\leq4(CM)^{2}L[\frac{2|\mu|^{-1/\alpha}\pi}{\alpha\sin(\pi/\alpha)}+5+2b\frac{|\mu|^{-1/\alpha }\pi}{\alpha\sin(\pi/\alpha)}]\int_{0}^{t}\frac{\tilde{C}(\tilde{M}s)^{n}}{n!}\mathrm{d}s\\
  &&=\frac{\tilde{C}(\tilde{M}t)^{n+1}}{(n+1)!}.
 \end{eqnarray*}
That is, (\ref{indu}) holds for $n+1.$ By induction, we get that (\ref{indu}) holds for all $n\geq 0.$ Furthermore, we find that
\begin{eqnarray*}
E\sup_{0\leq t\leq T}|x_{n+1}-x_{n}(t)|^{2}
&\leq& \tilde{M}\int_{0}^{T}\|x_{n}(s)-x_{n-1}(s)\|^{2}\mathrm{d}s\\
&\leq& 4\tilde{M}\int_{0}^{T}\frac{C[\tilde{M}s]^{n-1}}{(n-1)!}\mathrm{d}s=4\frac{C[\tilde{M}T]^{n}}{n!}.
\end{eqnarray*}
Hence
$$P\{\sup_{0\leq t\leq T}|x_{n+1}(t)-x_{n}(t)|>\frac{1}{2^{n}}\}\leq 4\frac{\tilde{C}[\tilde{M}T]^{n}}{n!}.$$
Note that $\sum_{n=0}^{\infty}4\frac{\tilde{C}[\tilde{M}T]^{n}}{n!}<\infty,$
by using the Borel-Cantelli lemma, we can get a stochastic process $x(t)$ on $[0,T]$ such that $x_{n}(t)$ uniformly converges to $x(t)$ as $n\rightarrow\infty$ almost surely.

It is easy to check that $x(t)$ is a unique mild solution of (\ref{1}). The proof of the theorem is complete.
\end{proof}
\section{The existence of the square-mean $S$-asymptotically $\omega$-periodic solution }
\begin{lemma}\label{le1}
If $x(t)\in SAP_{\omega}(L^{2}(P,H)),$ where $\omega\in\mathbf{Z}^{+},$ then $x([t])\in SAP_{\omega}(L^{2}(P,H)).$
\end{lemma}
\begin{proof}
The proof is similar to that of Lemma 2 in \cite{dimbour2014s}, for the self-contained, we give the proof.
Since $x(t)\in SAP_{\omega}(L^{2}(P,H)),$ then for any $\epsilon>0,$ $\exists~ T_{\epsilon}^{0}\in \mathbf{R}^{+},$ such that for any $t>T_{\epsilon},$ we have $\|x(t+\omega)-x(t)\|<\epsilon.$ Let $T_{\epsilon}=[T_{\epsilon}^{0}]+1.$ For $t>T_{\epsilon},$ we have $[t]\geq T_{\epsilon}$ for $T_{\epsilon}$ is an integer. Then we deduce that  for the above $\epsilon,$ $\exists ~T_{\epsilon}\in \mathbf{R}^{+},$ such that $\|x([t]+\omega)-x([t])\|=\|x([t+\omega])-x([t])\|<\epsilon.$
\end{proof}
\begin{lemma}\label{le2}
If $x(t)\in SAP_{\omega}(L^{2}(P,H)$ and $T(t-s)\in \mathcal{L}(\mathbf{R}^{+},\mathbf{R})$ then $\Gamma_{1}(t)=\int_{0}^{t}T(t-s)x(s)\mathrm{d}s\in SAP_{\omega}(L^{2}(P,H).$
\end{lemma}
The proof process is similar to that of Lemma 1 in \cite{dimbour2012s}, so we omit it.
\begin{lemma}\label{le3}
If $x(t)\in SAP_{\omega}(L^{2}(P,H)),$ then $$\Gamma_{2}(t)=\int_{0}^{t}S_{\alpha}(t-s)x(s)\mathrm{d}w(s)\in SAP_{\omega}(L^{2}(P,H)).$$
\end{lemma}
\begin{proof}It is obvious that $\Gamma_{2}(t)$ is $L^{2}$-continuous.
Since
\begin{eqnarray*}
&&\|\Gamma_{2}(t+\omega)-\Gamma_{2}(t)\|^{2}\\
&& =\|\int_{0}^{t+\omega}S_{\alpha}(t+\omega-s)x(s)\mathrm{d}w(s)- \int_{0}^{t}S_{\alpha}(t-s)x(s)\mathrm{d}w(s)\|^{2}\\
&&=2\|\int_{0}^{\omega}S_{\alpha}(t+\omega-s)x(s)\mathrm{d}w(s)\|^{2}\\&&~~+2\|\int_{\omega}^{t+\omega}S_{\alpha}(t+\omega-s)x(s)\mathrm{d}w(s)-\int_{0}^{t}S_{\alpha}(t-s)x(s)\mathrm{d}w(s)\|^{2}\\
&&\leq2\|\int_{0}^{\omega}S_{\alpha}(t+\omega-s)x(s)\mathrm{d}w(s)\|^{2}\\&&~~+2\|\int_{0}^{t}S_{\alpha}(t-s)(x(s+\omega)-x(s))\mathrm{d}w(s)\|^{2}.\end{eqnarray*}
Since $x(t)\in SAP_{\omega}(L^{2}(P,H)),$ for any $\epsilon>0,$ we can choose $T_{\epsilon}>0$ such that when $t>T_{\epsilon},$ $\|x(t+\omega)-x(t)\|<\epsilon.$
For the above $\epsilon,$ we have
\begin{eqnarray*}
&&2\|\int_{0}^{t}S_{\alpha}(t-s)(x(s+\omega)-x(s))\mathrm{d}w(s)\|^{2}\\
&&\leq4\int_{0}^{T_{\epsilon}}\|S_{\alpha}(t-s)\|^{2}\|x(s+\omega)-x(s)\|^{2}\mathrm{d}s\\&&~~+4\int_{T_{\epsilon}}^{t}\|S_{\alpha}(t-s)\|^{2}\|x(s+\omega)-x(s)\|^{2}\mathrm{d}s.
\end{eqnarray*}
Note that\begin{eqnarray*}
     2\|\int_{0}^{\omega}S_{\alpha}(t+\omega-s)x(s)\mathrm{d}w(s)\|^{2}\leq 2\frac{(CM)^{2}}{1+|\mu|^{2}t^{2\alpha}}\int_{0}^{\omega} \|x(s)\|^{2}\mathrm{d}s\rightarrow0,~t\rightarrow\infty,
         \end{eqnarray*}
 that$$\int_{0}^{T_{\epsilon}}\|S_{\alpha}(t-s)\|^{2}\|x(s+\omega)-x(s)\|^{2}\mathrm{d}s\leq 4\frac{(CM)^{2}}{1+|\mu|^{2}(t-T_{\epsilon})^{2\alpha}}\|x\|_{\infty}T_{\epsilon}\rightarrow0,~t\rightarrow\infty,$$
 and that
$$\int_{T_{\epsilon}}^{t}\|S_{\alpha}(t-s)\|^{2}\|x(s+\omega)-x(s)\|^{2}\mathrm{d}s\leq \epsilon^{2} \frac{(CM)^{2}|\mu|^{-2/\alpha}\pi}{2\alpha\sin(\pi/2\alpha)},$$
we get $\lim_{t\rightarrow\infty}\|\Gamma_{2}(t+\omega)-\Gamma_{2}(t)\|=0.$ So $\Gamma_{2}(t)\in SAP_{\omega}(L^{2}(P,H)).$
\end{proof}
The following lemma is obvious by using Lemma \ref{l1}, Lemma \ref{l2} and the similar discussion as that for Lemma \ref{le3}.
\begin{lemma}\label{le4}
If $x(t)\in SAP_{\omega}(L^{2}(P,H))$ and $F:\mathbf{R}^{+} \times L^{2}(P,H)\times U\rightarrow L^{2}(P,H)$ is uniformly Poisson square-mean S-asymptotically $\omega$-periodic in $t$ on bounded sets of $L^{2}(P,H),$ then $$\Gamma_{3}(t)=\int_{0}^{t}\int_{|u|_{U}<1}S_{\alpha}(t-s)F(s,x(s),u)\tilde{N}(\mathrm{d}u,\mathrm{d}s)\in SAP_{\omega}(L^{2}(P,H))$$ and $$\Gamma_{4}(t)=\int_{0}^{t}\int_{|u|_{U}\geq1}S_{\alpha}(t-s)G(s,x(s),u)N(\mathrm{d}u,\mathrm{d}s)\in SAP_{\omega}(L^{2}(P,H)).$$
\end{lemma}
\begin{theorem}
Assume that (H1)-(H2) are satisfied and $$f:\mathbf{R}^{+}\times L^{2}(P,H)\times L^{2}(P,H)\rightarrow L^{2}(P,H),$$ $$g:~\mathbf{R}^{+}\times L^{2}(P,H)\times L^{2}(P,H)\rightarrow \mathcal{L}(U,L^{2}(P,H))$$ are uniformly square-mean S-asymptotically $\omega$-periodic in $t$ on bounded sets of $L^{2}(P,H)\times L^{2}(P,H)$. Let $\omega \in \mathbf{Z}^{+}.$ $F:\mathbf{R}^{+}\times L^{2}(P,H)\times U\rightarrow L^{2}(P,H),$ $G:\mathbf{R}^{+}\times L^{2}(P,H)\times U\rightarrow L^{2}(P,H)$ are uniformly Poisson square-mean S-asymptotically $\omega$-periodic in $t$ on bounded sets of $L^{2}(P,H)\times L^{2}(P,H)$.
Then (\ref{1}) has a unique square-mean S-asymptotically $\omega$-periodic solution if $$2CM\{L\big(2\frac{|\mu|^{-1/\alpha}\pi}{\alpha\sin(\pi/\alpha)}+5\frac{|\mu|^{-2/\alpha}\pi}{2\alpha\sin(\pi/2\alpha)}+2b(\frac{|\mu|^{-1/\alpha}\pi}{\alpha\sin(\pi/\alpha)})^{2}\big)\}^{\frac{1}{2}}<1.$$
\end{theorem}
\begin{proof}
Define an operator $\tilde{\Gamma}:SAP_{\omega}(L^{2}(P,H))\mapsto SAP_{\omega}(L^{2}(P,H))$
\begin{eqnarray*}
&&(\tilde{\Gamma} x)(t)\\
&&=S_{\alpha}(t)c_{0}+\int_{0}^{t}S_{\alpha}(t-s)f(s,x([s]),x(s))\mathrm{d}s+\int_{0}^{t}S_{\alpha}(t-s)g(s,x([s]),x(s))\mathrm{d}w(s)
\\&&~+\int_{0}^{t}\int_{|u|<1}S_{\alpha}(t-s)F(s,x(s^{-}),u)\tilde{N}(\mathrm{d}u,\mathrm{d}s)\\
&&~+\int_{0}^{t}\int_{|u|\geq1}S_{\alpha}(t-s)G(s,x(s^{-}),u)N(\mathrm{d}u,\mathrm{d}s)
\end{eqnarray*}
for every $x\in SAP_{\omega}(L^{2}(P,H)).$ By (H1) and (\ref{sem}), we get that the operator $\tilde{\Gamma}$ is well defined. By Lemma \ref{le1}, Lemma \ref{le2}, Lemma \ref{le3} and Lemma \ref{le4}, we get $(\tilde{\Gamma} x)(t)\in SAP_{\omega}(L^{2}(P,H)).$
For every $x,y\in SAP_{\omega}(L^{2}(P,H)),$ we have
\begin{eqnarray*}
&&\|(\tilde{\Gamma} x)(t)-(\tilde{\Gamma} y)(t)\|^{2}\\
&&=\|\int_{0}^{t}S_{\alpha}(t-s)\big(f(s,x([s]),x(s))-f(s,y([s]),y(s))\big)\mathrm{d}s\\
&&~+\int_{0}^{t}S_{\alpha}(t-s)\big(g(s,x([s]),x(s))-g(s,y([s]),y(s))\big)\mathrm{d}w(s)\\
&&~+\int_{0}^{t}\int_{|u|<1}S_{\alpha}(t-s)\big(F(s,x(s^{-}),u)-F(s,y(s^{-}),u)\big)\tilde{N}(\mathrm{d}u,\mathrm{d}s)\\
&&~+\int_{0}^{t}\int_{|u|<1}S_{\alpha}(t-s)\big(G(s,x(s^{-}),u)-G(s,y(s^{-}),u)\big)N(\mathrm{d}u,\mathrm{d}s)\|^{2}\\
&&\leq 4(CM)^{2}\frac{|\mu|^{-1/\alpha}\pi}{\alpha\sin(\pi/\alpha)}\int_{0}^{t}\frac{1}{1+|\mu|(t-s)^{\alpha}}\|[f(s,x([s]),x(s))-f(s,y([s]),y(s))]\|^{2}\mathrm{d}s\\
&&+4\int_{0}^{t}\frac{(CM)^{2}}{1+|\mu|^{2}(t-s)^{2\alpha}}E\|(g(s,x([s]),x(s))-g(s,y([s]),y(s)))Q^{1/2}\|_{\mathcal{L}(U,L^{2}(P,H))}^{2}\mathrm{d}s\\
&&+4(CM)^{2}L\int_{0}^{t}\int_{|u|_{U}<1}\frac{1}{1+|\mu|^{2}(t-s)^{2\alpha}}\|F(s,x(s^{-}),u)-F(s,y(s^{-}),u)\|^{2}\nu(\mathrm{d}u)\mathrm{d}s\\
&&+8\|\int_{0}^{t}\int_{|u|_{U}\geq1}S_{\alpha}(t-s)\big(G(s,x(s^{-}),u)-G(s,y(s^{-}),u)\big)\tilde{N}(\mathrm{d}u,\mathrm{d}s)\|^{2}\\
&&+8\|\int_{0}^{t}\int_{|u|_{U}\geq1}S_{\alpha}(t-s)\big(G(s,x(s^{-}),u)-G(s,y(s^{-}),u)\big)\nu(\mathrm{d}u)\mathrm{d}s\|^{2}\\
&&\leq 4(CM)^{2}\frac{|\mu|^{-1/\alpha}\pi}{\alpha\sin{\pi/\alpha}}\int_{0}^{t}\frac{1}{1+|\mu|(t-s)^{\alpha}}L(\|x([s])-y([s])\|^{2}+\|x(s)-y(s)\|^{2})\mathrm{d}s\\
&&~+4(CM)^{2}\int_{0}^{t}\frac{1}{1+|\mu|^{2}(t-s)^{2\alpha}}L(\|x([s])-y([s])\|^{2}+\|x(s)-y(s)\|^{2})\mathrm{d}s\\
&&~+4(CM)^{2}L\int_{0}^{t}\frac{1}{1+|\mu|^{2}(t-s)^{2\alpha}}\mathrm{d}s\sup_{s\in\mathbf{R}^{+}}\|x(s)-y(s)\|^{2}\\
&&~+8(CM)^{2}L\int_{0}^{t}\frac{1}{1+|\mu|^{2}(t-s)^{2\alpha}}\mathrm{d}s\sup_{s\in\mathbf{R}^{+}}\|x(s)-y(s)\|^{2}\\
&&~+8(CM)^{2}b\int_{0}^{t}\frac{1}{1+|\mu|(t-s)^{\alpha}}\mathrm{d}s\\&&~~\quad\times\int_{0}^{t}\int_{|u|\geq1}\frac{1}{1+|\mu|(t-s)^{\alpha}}\|G(s,x(s^{-}),u)-G(s,y(s^{-}),u)\|^{2}\nu(\mathrm{d}u)\mathrm{d}s
\\
&&\leq 4(CM)^{2}L\big(2\frac{|\mu|^{-1/\alpha}\pi}{\alpha\sin(\pi/\alpha)}+5\frac{|\mu|^{-2/\alpha}\pi}{2\alpha\sin(\pi/2\alpha)}+2b(\frac{|\mu|^{-1/\alpha}\pi}{\alpha\sin(\pi/\alpha)})^{2}\big)\sup_{s\in\mathbf{R}^{+}}\|x(s)-y(s)\|^{2}.
\end{eqnarray*}
 Since $2CM\{L\big(2\frac{|\mu|^{-1/\alpha}\pi}{\alpha\sin(\pi/\alpha)}+5\frac{|\mu|^{-2/\alpha}\pi}{2\alpha\sin(\pi/2\alpha)}+2b(\frac{|\mu|^{-1/\alpha}\pi}{\alpha\sin(\pi/\alpha)})^{2}\big)\}^{\frac{1}{2}}<1 ,$ we obtain the result by the Banach contraction mapping principle.
\end{proof}

\end{document}